\documentclass[12pt, twoside, leqno]{article}
 \pdfoutput=1
\usepackage{amsmath,amsthm}
\usepackage{amssymb,latexsym}
\usepackage{enumerate}
\usepackage{lscape}
\usepackage{booktabs}

\theoremstyle{definition}

\newtheorem{teo}{Theorem}[section]

\newtheorem{lemmino}[teo]{Lemma}
\newtheorem{prop}[teo]{Proposition}

\numberwithin{equation}{section}

\begin{document}

\baselineskip=17pt

\title{Classification of Algebraic Function Fields with Class Number One}

\author{Pietro Mercuri
\\
Claudio Stirpe}

\date{ }

\maketitle

\renewcommand{\thefootnote}{}



\renewcommand{\thefootnote}{\arabic{footnote}}
\setcounter{footnote}{0}

\begin{abstract}
In this paper we prove that there are exactly eight function fields, up to isomorphism,
over finite fields with class number one and positive genus. This classification was already suggested,
although not completely proved, in a previous work about this topic (see Stirpe~\cite{st1}).
\end{abstract}

\section{Introduction}
The problem of the determination of algebraic function fields over finite fields with class number one was already treated in the paper of Leitzel, Madan and Queen \cite{lemaqu}. But, in Stirpe \cite{st1}, one more example is given, showing that the previous classification is incomplete. In this paper we give the full list of function fields with class number one and positive
genus 
 hence their classification 
 is now complete.

The list of all function fields with class number one and positive genus is given in the following theorem.

\begin{teo}\label{main} 
Let $K$ be a function field over the finite field $\mathbb{F}_q$ of genus $g>0$ with class number one.
Then $K$ is isomorphic to a function field $\mathbb{F}_q(x,y)$ defined by one of the following equations:\\
(i) $y^2+y+ x^3 + x + 1=0$, with $g=1$ and $q=2$;\\
(ii) $y^2 + y+ x^5 + x^3 + 1=0$, with $g=2$ and $q=2$;\\
(iii) $y^2+y+ (x^3 + x^2 + 1)/(x^3 + x + 1)=0$, with $g=2$ and $q=2$;\\
(iv) $y^4+xy^3+(x^2+x)y^2+(x^3+1)y+x^4+x+1=0$, with $g=3$ and $q=2$;\\
(v) $y^4+(x^3+x+1)y+x^4+x+1=0$, with $g=3$ and $q=2$;\\
(vi) $y^2+2 x^3 + x + 1=0$, with $g=1$ and $q=3$;\\
(vii) $y^2+y-x^3+\alpha=0$, with $g=1$ and $q=4$,
where $\alpha$ is a generator of the multiplicative group $\mathbb{F}^*_4$;\\
(viii) $y^5+y^3+y^2(x^3+x^2+x)+y(x^7+x^5+x^4+x^3+x)/(x^4+x+1)+ (x^{13}+x^{12}+x^8+x^6+x^2+x+1)/(x^4+x+1)^2=0$, with $g=4$ and $q=2$.\\
\end{teo}

Function fields (i-vii) are already given in \cite{lemaqu} but in that paper the authors wrongly claimed to have ruled out the
only other possibility of curves of genus 4 over $\mathbb{F}_2$. Madan and Queen also showed in \cite{maqu}
that any other example of function field with class number one and positive genus should be a function field of genus 4 with
exactly one place of degree 4 and without places of smaller degree. A function field with this properties was later found in
\cite{st1} but the author did not prove that such example is unique up to isomorphism. This function field is listed in (viii)
and we prove uniqueness in Section 3 giving a definitive proof of Theorem \ref{main}.
 The proof is elementary and requires only basic facts about function fields and Class Field Theory.
The necessary background is given in the next section.

In the fourth section we show another proof that is a complete version of the argument in \cite{lemaqu}.
The interested reader can also find a simplified argument in Qibin Shen and Shuhui Shi \cite{shen}
which was done independently around the same time.

\section{Background and definitions}
Given two extensions of function field $K$ and $L$ of $\mathbb{F}_q(x)$, we assume that they are embedded in a same
algebraic closure of $\mathbb{F}_q(x)$.

Let $K/F$ be a separable extension of function fields. Let $n$ be the degree $[K:F]$ and let $g_K$ and $g_F$ be the genus of $K$ and $F$, respectively. We assume that the constant fields of $F$ and $K$ are the same. Then the genera of $F$ and $K$
are related by the Hurwitz Genus Formula:
\begin{equation} \label{hur0}
2g_K-2=n(2g_F-2)+ \deg \mathrm{Diff}(K/F),
\end{equation}
where $\mathrm{Diff}(K/F)=\sum_P d_P P$ is the different of $K/F$ and the sum runs over the ramified places of $K$.
The reader can see Stichtenoth \cite[Chapter 3]{st93} for definitions and main results.
For our purposes we need only to know that 
the sum is actually a finite sum and the coefficients $d_P$ are positive integers satisfying
the following inequality, known as Dedekind's Different Theorem (see \cite[Theorem 3.5.1]{st93}):
\begin{equation} \label{dede}
d_P\geq e_P-1,
\end{equation}
where $e_P$ is the ramification index of $P$ in $K/F$ and equality holds if
and only if the characteristic $p$ of the finite field does not divide $e_P$.

Let $P$ be a place of $F$ and let $P_1$, $P_2$, $\ldots$, $P_r$ be the places of $K$ over~$P$.
We denote by $e(P_i|P)$ and $f(P_i|P)$, respectively, the ramification index and the inertia degree
of $P_i/P$ for $i=1,2,\ldots$, $r$. The integers $e(P_i|P)$ and $f(P_i|P)$ are related by the 
following equality 
(see \cite[Theorem 3.1.11]{st93})
\begin{equation} \label{fund}
\sum_{i=1}^r f(P_i|P)e(P_i|P)=n.
\end{equation}

We also need the following lemma.
\begin{lemmino}[Abhyankar's Lemma] \label{abhya}
Let $K_1/F$ and $K_2/F$ be separable extensions of function fields and let $p$ be the characteristic of $F$.
Let $K=K_1K_2$ be the compositum of $K_1$ and $K_2$.
Let $Q$ be a place of $K$ and let $P$ be the place $Q\cap F$. We denote by $P_1$ and $P_2$, respectively,  the places $Q\cap K_1$ and $Q\cap K_2$ and by $e_1$ and $e_2$ the ramification index of $P_1$ and $P_2$ in $K_1/F$ and $K_2/F$. We assume that $p$ does not divide both integers $e_1$ and $e_2$.
Then the ramification index $e_Q$ of $Q/P$ in $K/F$ is given by $$e_Q=\mathrm{lcm}(e_1,e_2).$$
\end{lemmino}

\begin{proof}
See \cite[Theorem 3.9.1]{st93}.
\end{proof}
 
Finally, we recall a result of Class Field Theory. Let $\mathfrak{f}$ be a place of $K$ and $S$ be a non-empty set of places of $K$ not containing $\mathfrak{f}$. 
We denote by $K^\mathfrak{f}_S$ the maximal abelian function field extension of $K$ such that $K^\mathfrak{f}_S/K$ is
unramified outside $\mathfrak{f}$ and the points of $S$ are totally split.
It is known that $K^\mathfrak{f}_S$ is a finite extension. The reader can find more details in Auer \cite{au00}.
The following result is a direct consequence of Stirpe \cite[Remark 3.3]{st0} concerning enumeration of cyclic ray class field extensions with given conductor $\mathfrak{f}$.

\begin{prop}\label{artintate}
Let $\mathfrak{m}$ be a place of $K$ of degree $t$. Let $d$ be a divisor of $\frac{h_K}{q-1}(q^{t}-1)$, where $h_K$
is the class number of $K$. Then there are exactly $d$ cyclic extensions of $K$ of degree $d$
with conductor dividing $\mathfrak{m}$ and constant field $\mathbb{F}_q$.
\end{prop}

\section{Uniqueness in the genus 4 case}
From now on, we denote by $L$ the function field listed in Theorem \ref{main} part (viii). 
The function field $L$ has genus 4 and class number one by \cite{st1}. Let $K$ denote a function
field over the finite field $\mathbb{F}_2$ of genus 4 and class number one, we want to prove that $K$ is isomorphic to $L$.
In \cite{maqu} the zeta function of $K$ has been computed and the authors show that $K$ has exactly three places of
degree 5, one place of degree 4 and no places of smaller degree.

First, we consider the degree $d=[K:\mathbb{F}_2(x)]$. The following Lemma allows us to assume that $d=5$. 

\begin{lemmino}
Let $K$ be a function field with the constant field $\mathbb{F}_2$ and with genus 4 and class number one. Then there is an element $x\in K$ such that the degree $d=[K:\mathbb{F}_2(x)]$ is equal to 5.
\end{lemmino}

\begin{proof}
Let $Q$ be a place of $K$ of degree 5. By Riemann-Roch Theorem
$$ l_K(Q)\geq \deg Q-g_K+1=2.$$
Let $x\in K$ be a non constant element in $L_K(Q)$. Then the pole divisor $(x)_\infty$ of $x$ is equal to $Q$ and so
(see \cite[Theorem 1.4.11]{st93})
$$d=[K:\mathbb{F}_2(x)]=\deg Q= 5.$$
\end{proof}

From now on, we set $x \in K$ as above so that $K$ is a separable extension of degree 5 of $\mathbb{F}_2(x)$.

By the Hurwitz Genus Formula we can compute the degree of the different of $K/\mathbb{F}_2(x)$
\begin{equation}\label{diff}
\deg \mathrm{Diff}(K/\mathbb{F}_2(x))=2g_K-2+2d=16.
\end{equation}
By Dedekind's Different Theorem $d_Q\geq e_Q-1$  for any ramified place $Q$
of $K$ but equality in (\ref{dede}) holds only if $e_Q$ is odd so $d_Q\geq 2$.
It follows that $\sum_Q \deg Q\leq 8$, but there is only one place
of $K$ of degree 4 and all other places have higher degree. It follows that there is only one ramified place in
$K/\mathbb{F}_2(x)$. We denote by $Q$ such a place of $K$ and by $P$ the place $Q\cap \mathbb{F}_2(x)$.

\begin{lemmino}
The ramified place $P$ in $K/\mathbb{F}_2(x)$ has degree 4 and the function field $K$ is a covering of degree 5 of $\mathbb{F}_2(x)$ totally ramified at $P$.
\end{lemmino}

\begin{proof}
By equality (\ref{diff}) the place $Q$ over $P$ satisfies $d_Q\deg Q=16$. It follows that either $d_Q=2$ and $\deg Q=8$ or $d_Q=4$ and $\deg Q=4$.

We show that the first case cannot occur.
We assume that $d_Q=2$, hence by Dedekind's Theorem $e_Q\leq d_Q+1=3$ so either $e_Q=2$ or $e_Q=3$. We denote by $u\in \mathbb{F}_2[x]$ the irreducible monic polynomial with zero divisor $Z_u$ equal to $P$.

STEP 1 (Case $e_Q=3$). 
We consider the extension of constant field $\mathbb{F}_4(x)/\mathbb{F}_2(x)$.
The place $P$ has even degree because $\deg Q=8$, hence it splits in two places $P_1$ and $P_2$. We consider a Kummer extension $E/\mathbb{F}_4(x)$ of degree~3 of $\mathbb{F}_4(x)$ totally ramified at $P_i$ for $i=1$, $2$ and unramified outside $P_1$ and $P_2$ (see \cite[Proposition 3.7.3]{st93}). We can define explicitly $E$ as $E=\mathbb{F}_4(x,t)$ where $t$ is such that $t^3+u=0$ (the interested reader can see \cite[Corollary 3.7.4]{st93}). Now, we consider the compositum $EK$. By Abhyankar's Lemma the extensions of function fields $EK/K\mathbb{F}_4$ and $EK/E$ are both unramified and the constant field of $EK$ is 
$\mathbb{F}_4$ because
\begin{equation}\label{deg1}
[EK:K\mathbb{F}_4]=[E:\mathbb{F}_4(x)]=3,
\end{equation}
is coprime to 
\begin{equation}\label{deg2}
[EK:E]=[K\mathbb{F}_4:\mathbb{F}_4(x)]=5.
\end{equation}
We can compute the genus of $EK$ in two ways by substitution of (\ref{deg1}) and (\ref{deg2}) in the Hurwitz Formula. We obtain that
$$2g_{EK}-2=[EK:K\mathbb{F}_4](2g_K-2)=3\cdot 6=18,$$
and, similarly,
$$2g_{EK}-2=[EK:E](2g_E-2)=5(2g_E-2).$$ 
This leads to a contradiction $5(2g_E-2)=18$.

STEP 2 (Case $e_Q=2$). The proof is similar to the previous case although is more complicated because
the extension $K/\mathbb{F}_2(x)$ is wild. We consider an Artin-Schreier extension $A/\mathbb{F}_2(x)$
of degree 2 of $\mathbb{F}_2(x)$ ramified at $P$ and unramified outside $P$. 
We can define explicitly $A$ as $A=\mathbb{F}_2(x,t)$ with $t^2+t+u^{-1}=0$ 
(see \cite[Remark 3.7.9, part a]{st93}). By classical theory of Artin-Schreier extensions, 
the exponent of the different $d_{P'}$ of the place $P'$ of $A$ over $P$ is equal to $(p-1)(m_P+1)=2$
(see \cite[Proposition 3.7.8, part c]{st93}), where $m_P=-v_P(u^{-1})=1$.
Again we consider the compositum $AK$. As in Step 1 the constant function field of $AK$ is $\mathbb{F}_2$ because
$$[AK:K]=[A:\mathbb{F}_2(x)]=2,$$
is coprime to 
$$[AK:A]=[K:\mathbb{F}_2(x)]=5.$$
By the Hurwitz Formula we get 
\begin{equation}\label{contrad1}
2g_{AK}-2=[AK:K](2g_K-2)+\sum_{T|P} d(T|T_K)\deg T,
\end{equation}
and 
\begin{equation}\label{contrad2}
2g_{AK}-2=[AK:A](2g_A-2)+\sum_{T|P} d(T|T_A) \deg T,
\end{equation}
where the sums run over the places $T$ of $AK$ over $P$ and where $T_K$ is the place $T\cap K$ and $T_A=T\cap A$.
Since the ramification indices in $K/\mathbb{F}_2(x)$ and in $A/\mathbb{F}_2(x)$ are equal, then the ramification indices $e(T|T_A)$ and $e(T|T_K)$ are equal too. Now, we prove that the coefficients $d(T|T_A)$ and $d(T|T_K)$ are also equal
for any place $T$ of $AK$ over $P$. By the transitivity of the different (see \cite[Corollary 3.4.12 b]{st93}))
$$d(T|P)=e(T|T_A)d(T_A|P)+d(T|T_A)=e(T|T_K)d(T_K|P)+d(T|T_K),$$
 follows that
\[
d(T|T_K)=d(T|T_A),
 \]
because
\[
e(T|T_K)=e(T|T_A)
\]
and
\[
d(T_K|P)=d(T_A|P)=2.
\]
By equations (\ref{contrad1}) and (\ref{contrad2}) we get a contradiction because
\[
[AK:A]=5
\]
does not divide
\[
[AK:K](2g_K-2)=12.
\]

It follows that the only possible case is $d_Q=4$ and $\deg Q=4$.

Finally, we show that $e_Q=5$. If the ramification is wild we can only have $e_Q=2$ or $e_Q=4$. As in Step 2 above, one can show that $e_Q\neq 2$ by considering $A=\mathbb{F}_2(x,t)$ with $t^2+t+u^{-3}=0$.
We assume that $e_Q=4$. By equality (\ref{fund}) we see that $f_Q<2$, therefore $f_Q=1$ and, similarly, any other place $P'$ of $K$ over $P$ is not partially inert, i.e. $f_{P'}=1$. But there is only one place of $K$ of degree 4, so $P'$ should be at least partially inert, contradiction. Hence $e_Q\neq 4$. Therefore the ramification must be tame, but in this case the equality holds in equation (\ref{fund}) and we get $e_Q=5$ and $\deg Q=\deg P=4$. 
\end{proof}

\begin{proof}[Proof of Theorem \ref{main}]
Consider the ramified place $P$ under $Q$ in $\mathbb{F}_2(x)$. In the following, without loss of generality we can assume that $P$ is the place $\mathfrak{f}=(x^4+x+1)$. In fact, if $P$ is the place $(x^4+x^3+1)$, we can use an isomorphism of the rational function field, namely $\sigma(x)=1/x$, to send $P$ in $\mathfrak{f}$. Similarly, if $P=(x^4+x^3+x^2+x+1)$ we can choose $\sigma(x)=1/(x+1)$ and again $\sigma(P)=\mathfrak{f}$.
As before we denote by $L$ the function field listed in \ref{main} part (viii).
One can show that $L/\mathbb{F}_2(x)$ is a cyclic extension (see \cite{st1}), totally ramified in $\mathfrak{f}$ with ramification index equal to 5. We show that $K/\mathbb{F}_2(x)$ is also a Galois extension.

We consider the compositum $KL$. By Abhyankar's Lemma \ref{abhya} the field extension $KL/\mathbb{F}_2(x)$ is partially ramified
in $\mathfrak{f}$ with ramification index equal to 5. It follows that $KL/K$ is unramified. But the class number of $K$ is one by hypothesis and $KL/K$ is a Galois extension because $L/\mathbb{F}_2(x)$ is Galois hence $KL\subseteq K\mathbb{F}_{2^5}$. It follows that either $K=L$ or $KL=K\mathbb{F}_{2^5}$. In the first case we are done, in the second one we get
$$L\mathbb{F}_{2^5}= K\mathbb{F}_{2^5}.$$
It follows that the Galois closure of $K$ is contained in $K\mathbb{F}_{2^5}$. The Galois group 
$H=\mathrm{Gal}(\mathbb{F}_{2^5}/\mathbb{F}_{2})$
is canonically embedded in the Galois group 
$G=\mathrm{Gal}(K\mathbb{F}_{2^5}/\mathbb{F}_{2}(x))$ and it is normal.
By Galois Theory, $K$ is a normal extension with Galois group $G/H$ and degree
$$[K:\mathbb{F}_2(x)]=|G/H|=25/5=5.$$
By Class Field Theory, $K$ is a ray class field extension over $\mathbb{F}_{2}(x)$
of degree 5 and genus 4 with conductor $\mathfrak{f}$.
By Proposition \ref{artintate} there are only five of such extensions. These five function fields are 
explicitly constructed in \cite{st1}
but three of them are subfields of the ray class field extensions $K_S^\mathfrak{f}/\mathbb{F}_{2}(x)$ with $S$
given by a rational place of $\mathbb{F}_{2}(x)$: these function fields have rational places and so their class number
is greater than one by Madan and Queen \cite[Theorem 2]{maqu}. 
The last two ray class fields are contained in the ray class field extensions $K_S^\mathfrak{f}/\mathbb{F}_{2}(x)$, where $S$ is a place of $\mathbb{F}_{2}(x)$ of degree 7 given by $P_1=(x^7+x^4+1)$ or $P_2=(x^7+x^3+1)$. These ray class fields have class number one and they are isomorphic to each other. The first choice $S=P_1$ gives the function field $L$. The reader can see Stirpe~\cite[Remark 3]{st1}, for more details and explicit computations. In both cases $K$ is a ray class field extension of $\mathbb{F}_{2}(x)$ isomorphic to $L$.
\end{proof}

\section{Table of the original proof}
In the proof given by Leitzel, Madan and Queen in \cite[Section 2]{lemaqu}, an enumerative argument is given to prove that there are no function fields with class number one and genus 4. They prove that the field of genus 4 and class number one we are looking for, must be defined by a cubic and a quadric. They also prove that this pair must be one of the following:
\begin{align*}
(C_1,Q_1+L(k_1,k_2,k_3,k_4)^2),\\
(C_2,Q_2+L(k_1,k_2,k_3,k_4)^2),\\
(C_3,Q_3+L(k_1,k_2,k_3,k_4)^2),\\
(C_4,Q_4+L(k_1,k_2,k_3,k_4)^2),
\end{align*}
where
\begin{align*}
C_1&=x_2^3+x_1x_3^2+x_4^3+x_1^2x_3+x_3x_4^2, \\
C_2&=x_2^3+x_1x_3^2+x_2^2x_3+x_2^2x_4+x_1^3+x_3^2x_4+x_1^2x_2+x_2x_4^2, \\
C_3&=x_2^2x_3+x_1x_4^2+x_3^3+x_3^2x_4+x_1^2x_2+x_4^3+x_1^2x_3+x_3x_4^2, \\
C_4&=x_1^3+x_1^2x_3+x_1x_4^2+x_2^2x_4+x_2x_4^2+x_3^3+x_3x_4^2+x_4^3, \\
Q_1&=x_1x_2+x_3x_4, \\
Q_2&=x_1x_2+x_1x_3+x_1x_4+x_2x_4, \\
Q_3&=x_1x_3+x_2x_3+x_2x_4+x_3x_4, \\
Q_4&=x_1x_4+x_2x_3+x_3x_4\\
\end{align*}
and where
\[
L(k_1,k_2,k_3,k_4)=k_1x_1+k_2x_2+k_3x_3+k_4x_4,
\]
is a homogeneous linear form in $x_1,x_2,x_3,x_4$ with
$k_1,k_2,k_3,k_4\in\{0,1\}$. The list given at page 15 in \cite{lemaqu} shows a list of only eight possible quadric surfaces and all of them are discarded by the authors case by case. This list is actually incomplete. In the following table, we write down
the complete list and we show in the second column the rational points of lowest degree that allows us to discard each case but one. This case not discarded, i.e. the pair $(C_2,Q_2+L(1,0,1,1)^2)$, defines a field isomorphic to that one of Theorem \ref{main}, case viii. This argument gives a second proof of Theorem \ref{main}.

The complete list, below, shows all the 64 quadric. The interested reader can also find a simplified proof in \cite{shen},
where the number of cases to check is reduced by Qibin Shen and Shuhui Shi to 24 cases only. The authors also give minor simplification of the original paper \cite{lemaqu}.

\subsection*{Acknowledgements}
The authors are grateful to Dinesh Thakur for his helpful comments
and for letting us know about the paper of Qibin Shen and Shuhui Shi \cite{shen},
where an argument similar to the one of Section 4 is given.

\footnotesize
\begin{center}{
Pietro Mercuri\\
mercuri.ptr@gmail.com;\\
Claudio Stirpe \\
clast@inwind.it.}
\end{center}
\scriptsize
\begin{landscape}
\begin{align*}
&\begin{array}{|cc|cc|}
\toprule
Q_1+L(k_1,k_2,k_3,k_4)^2 & \text{Low degree point} & Q_3+L(k_1,k_2,k_3,k_4)^2 & \text{Low degree point} \\
\midrule
x_1x_2 + x_3x_4 & (1 : 0 : 0 : 0) & x_1x_3 + x_2x_3 + x_2x_4 + x_3x_4 & (1 : 0 : 0 : 0) \\ 
x_1x_2 + x_3x_4 + x_4^2 & (1 : 0 : 0 : 0) & x_1x_3 + x_2x_3 + x_2x_4 + x_3x_4 + x_4^2 & (1 : 0 : 0 : 0)  \\
x_1x_2 + x_3^2 + x_3x_4 & (1 : 0 : 0 : 0) & x_1x_3 + x_2x_3 + x_3^2 + x_2x_4 + x_3x_4 & (1 : 0 : 0 : 0)  \\
x_1x_2 + x_3^2 + x_3x_4 + x_4^2 & (1 : 0 : 0 : 0) & x_1x_3 + x_2x_3 + x_3^2 + x_2x_4 + x_3x_4 + x_4^2 & (1 : 0 : 0 : 0)  \\
x_1x_2 + x_2^2 + x_3x_4 & (1 : 0 : 0 : 0) & x_2^2 + x_1x_3 + x_2x_3 + x_2x_4 + x_3x_4 & (1 : 0 : 0 : 0)  \\
x_1x_2 + x_2^2 + x_3x_4 + x_4^2 & (1 : 0 : 0 : 0) & x_2^2 + x_1x_3 + x_2x_3 + x_2x_4 + x_3x_4 + x_4^2 & (1 : 0 : 0 : 0)  \\
x_1x_2 + x_2^2 + x_3^2 + x_3x_4 & (1 : 0 : 0 : 0) & x_2^2 + x_1x_3 + x_2x_3 + x_3^2 + x_2x_4 + x_3x_4 & (1 : 0 : 0 : 0)  \\
x_1x_2 + x_2^2 + x_3^2 + x_3x_4 + x_4^2 & (1 : 0 : 0 : 0) & x_2^2 + x_1x_3 + x_2x_3 + x_3^2 + x_2x_4 + x_3x_4 + x_4^2 & (1 : 0 : 0 : 0) \\ 
x_1^2 + x_1x_2 + x_3x_4 & (0 : 0 : 1 : 0) & x_1^2 + x_1x_3 + x_2x_3 + x_2x_4 + x_3x_4 & (0 : 1 : 0 : 0) \\ 
x_1^2 + x_1x_2 + x_3x_4 + x_4^2 & (0 : 0 : 1 : 0) & x_1^2 + x_1x_3 + x_2x_3 + x_2x_4 + x_3x_4 + x_4^2 & (0 : 1 : 0 : 0)  \\
x_1^2 + x_1x_2 + x_3^2 + x_3x_4 & (1 : 0 : 1 : 0) & x_1^2 + x_1x_3 + x_2x_3 + x_3^2 + x_2x_4 + x_3x_4 & (0 : 1 : 0 : 0)  \\
x_1^2 + x_1x_2 + x_3^2 + x_3x_4 + x_4^2 & (1 : 0 : 1 : 0) & x_1^2 + x_1x_3 + x_2x_3 + x_3^2 + x_2x_4 + x_3x_4 + x_4^2 & (0 : 1 : 0 : 0) \\ 
x_1^2 + x_1x_2 + x_2^2 + x_3x_4 & (0 : 0 : 1 : 0) & x_1^2 + x_2^2 + x_1x_3 + x_2x_3 + x_2x_4 + x_3x_4 & (1 : 1 : 1 : 1)  \\
x_1^2 + x_1x_2 + x_2^2 + x_3x_4 + x_4^2 & (0 : 0 : 1 : 0) & x_1^2 + x_2^2 + x_1x_3 + x_2x_3 + x_2x_4 + x_3x_4 + x_4^2 & (0 : 0 : 1 : 1) \\ 
x_1^2 + x_1x_2 + x_2^2 + x_3^2 + x_3x_4 & (1 : 0 : 1 : 0) &  x_1^2 + x_2^2 + x_1x_3 + x_2x_3 + x_3^2 + x_2x_4 + x_3x_4 & (0 : 0 : 1 : 1) \\ 
x_1^2 + x_1x_2 + x_2^2 + x_3^2 + x_3x_4 + x_4^2 & (1 : 0 : 1 : 0) & x_1^2 + x_2^2 + x_1x_3 + x_2x_3 + x_3^2 + x_2x_4 + x_3x_4 + x_4^2 & (1 : 1 : 1 : 1) \\ 
\bottomrule
Q_2+L(k_1,k_2,k_3,k_4)^2 & \text{Low degree point} & Q_4+L(k_1,k_2,k_3,k_4)^2 & \text{Low degree point} \\
\midrule
x_1x_2 + x_1x_3 + x_1x_4 + x_2x_4 & (0 : 0 : 1 : 0) & x_2x_3 + x_1x_4 + x_3x_4 & (0 : 1 : 0 : 0) \\
x_1x_2 + x_1x_3 + x_1x_4 + x_2x_4 + x_4^2 & (0 : 0 : 1 : 0) & x_2x_3 + x_1x_4 + x_3x_4 + x_4^2 & (0 : 1 : 0 : 0) \\
x_1x_2 + x_1x_3 + x_3^2 + x_1x_4 + x_2x_4 & (0 : 0 : 0 : 1) & x_2x_3 + x_3^2 + x_1x_4 + x_3x_4 & (0 : 1 : 0 : 0) \\
x_1x_2 + x_1x_3 + x_3^2 + x_1x_4 + x_2x_4 + x_4^2 & (1 : 1 : 1 : 1) & x_2x_3 + x_3^2 + x_1x_4 + x_3x_4 + x_4^2 & (0 : 1 : 0 : 0) \\
x_1x_2 + x_2^2 + x_1x_3 + x_1x_4 + x_2x_4 & (0 : 0 : 1 : 0) & x_2^2 + x_2x_3 + x_1x_4 + x_3x_4 & (1 : 1 : 1 : 1) \\
x_1x_2 + x_2^2 + x_1x_3 + x_1x_4 + x_2x_4 + x_4^2 & (0 : 0 : 1 : 0) & x_2^2 + x_2x_3 + x_1x_4 + x_3x_4 + x_4^2 & (\beta : 0 : \beta^3 : 1)\\
x_1x_2 + x_2^2 + x_1x_3 + x_3^2 + x_1x_4 + x_2x_4 & (0 : 0 : 0 : 1) & x_2^2 + x_2x_3 + x_3^2 + x_1x_4 + x_3x_4 & (1 : 0 : \alpha : 1)\\ 
x_1x_2 + x_2^2 + x_1x_3 + x_3^2 + x_1x_4 + x_2x_4 + x_4^2 & (1 : 0 : 1 : 0) & x_2^2 + x_2x_3 + x_3^2 + x_1x_4 + x_3x_4 + x_4^2 & (1 : 1 : 1 : 1) \\
x_1^2 + x_1x_2 + x_1x_3 + x_1x_4 + x_2x_4 & (0 : 0 : 1 : 0) & x_1^2 + x_2x_3 + x_1x_4 + x_3x_4 & (1 : 1 : 1 : 1) \\
x_1^2 + x_1x_2 + x_1x_3 + x_1x_4 + x_2x_4 + x_4^2 & (0 : 0 : 1 : 0) & x_1^2 + x_2x_3 + x_1x_4 + x_3x_4 + x_4^2 & (0 : 1 : 0 : 0) \\
x_1^2 + x_1x_2 + x_1x_3 + x_3^2 + x_1x_4 + x_2x_4 & (0 : 0 : 0 : 1) & x_1^2 + x_2x_3 + x_3^2 + x_1x_4 + x_3x_4 & (0 : 1 : 0 : 0) \\
x_1^2 + x_1x_2 + x_1x_3 + x_3^2 + x_1x_4 + x_2x_4 + x_4^2 & \text{Point of degree 4} & x_1^2 + x_2x_3 + x_3^2 + x_1x_4 + x_3x_4 + x_4^2 & (0 : 1 : 0 : 0) \\
x_1^2 + x_1x_2 + x_2^2 + x_1x_3 + x_1x_4 + x_2x_4 & (0 : 0 : 0 : 1) & x_1^2 + x_2^2 + x_2x_3 + x_1x_4 + x_3x_4 & (0 : \alpha : 1 : 1)\\ 
x_1^2 + x_1x_2 + x_2^2 + x_1x_3 + x_1x_4 + x_2x_4 + x_4^2 & (0 : 0 : 1 : 0) & x_1^2 + x_2^2 + x_2x_3 + x_1x_4 + x_3x_4 + x_4^2 & (1 : 1 : 1 : 1) \\
x_1^2 + x_1x_2 + x_2^2 + x_1x_3 + x_3^2 + x_1x_4 + x_2x_4 & (0 : 0 : 0 : 1) & x_1^2 + x_2^2 + x_2x_3 + x_3^2 + x_1x_4 + x_3x_4 & (1 : 1 : 1 : 1) \\
x_1^2 + x_1x_2 + x_2^2 + x_1x_3 + x_3^2 + x_1x_4 + x_2x_4 + x_4^2 & (1 : 1 : 1 : 1) & x_1^2 + x_2^2 + x_2x_3 + x_3^2 + x_1x_4 + x_3x_4 + x_4^2 & (0 : \alpha : 1 : 1)\\ 
\midrule
\bottomrule
\end{array} \\
&\begin{array}{c}
\text{Notation: we denote by }\alpha\in\mathbb{F}_4\text{ a solution of }x^2+x+1=0 \text{ and by }\beta\in\mathbb{F}_8\text{ a solution of }x^3+x+1=0.
\end{array}
\end{align*}

\end{landscape}


\small
\begin{thebibliography}{HD}
\bibitem{au00} R. Auer, \emph{Ray class fields of global function fields with many rational places}, Acta Arithmetica 95, 97-122 (2000).
\bibitem{lemaqu} J. Leitzel, M. Madan, C. Queen \emph{Algebraic function fields with small class number}, Journal of Number Theory 7, 11-27 (1975).
\bibitem{maqu} M. Madan, C. Queen \emph{Algebraic function fields of class number one}, Acta Arithmetica 20, 423-432 (1972).
\bibitem{shen} Q. Shen, S. Shi: \emph{Function fields of class number one}, Preprint (2014).
\bibitem{st93} H. Stichtenoth: \emph{Algebraic function fields and codes}. Berlin, Springer-Verlag, 2nd edition (2008).
\bibitem{st0} C. Stirpe \emph{An upper bound for the minimum genus of a curve without points of small degree}, Acta Arithmetica 160, 115-128  (2013).
\bibitem{st1} C. Stirpe \emph{A counterexample to 'Algebraic function fields with small class number'}, Journal of Number Theory 143, 402-404 (2014).
\end {thebibliography}
\end{document}